\documentclass[preprint,review,10pt]{amsart}
\usepackage{amsmath,amssymb,amsfonts,xspace}
\newtheorem{theorem}{Theorem}[section]

\theoremstyle{definition}
\newtheorem{definition}[theorem]{Definition}
\newtheorem{example}[theorem]{Example}

\newtheorem{proposition}[theorem]{Proposition}
\newtheorem{corollary}[theorem]{Corollary}

\theoremstyle{remark}

\numberwithin{equation}{section}

\begin{document}

\title{  quasi $S$-primary hyperideals }

\author{Mahdi Anbarloei}
\address{Department of Mathematics, Faculty of Sciences,
Imam Khomeini International University, Qazvin, Iran.
}

\email{m.anbarloei@sci.ikiu.ac.ir}





\begin{abstract}
In this paper, we define and study  quasi $S$-primary hyperideals, weakly quasi $S$-primary hyperideals and strongly quasi $S$-primary hyperideals.

\end{abstract}
\maketitle

\section{Introduction}
In recent years, the primary ideals and their exapnsions have a significant 
place in commutative algebra and they draw attention by many authors. Looking
at the respective studies, the notion of quasi $S$-primary ideals  was proposed by Celikel and Hamed in 2023. In their celebrated paper \cite{Celikel}, a proper ideal $I$ of a commutative ring $R$ disjoint with $S$ is said to be a quasi $S$-primary ideal where $S$ is a multiplicatively closed subset  of $R$ if there exists a $t \in S$ such that for all $x,y \in R$ if $xy \in I$, then $tx \in rad(I)$ or $ty \in rad(I)$. Afterwards, Guesimi generalized this concept by defining weakly quasi $S$-primary ideals \cite{Guesmi}. Moreover, Moulahi presented  an intermediate class between $S$-primary ideals and quasi $S$-primary ideals in \cite{Moulahi}.

The study of hyperstructures was initiated by Marty \cite{marty} and continued by many mathematicians \cite{f1,f2,f4,f5,x2,x1,f11,f7,f17,f8,f9}. Multiplicative hyperrings as  an important class of algebraic hyperstructures were proposed by Rota \cite{f14} where   the addition is an operation and the multiplication is a hyperoperation.  These hyperrings  have been   investigated in \cite{ameri, ameri5, ameri6,    Ciampi, Kamali,  f16}. 
More exactly,  the triple  $(G,+,\circ)$ is a  commutative multiplicative hyperring  if  {\bf I.} $(G,+)$ is a commutative group,  {\bf II.} $(G,\circ)$ is a semihypergroup; 
{\bf III.} $(-x)\circ y =x\circ (-y) =  -(x\circ y)$ for any $x, y \in G$,
{\bf IV.}  $(y+z)\circ x \subseteq y\circ x+z\circ x$ and  $x\circ (y+z) \subseteq x\circ y+x\circ z$ for any $x, y, z \in G $
{\bf V.} $x \circ y =y \circ x$ for any $x,y \in G$  \cite{f10}. The multiplicative hyperring $G$ is  strongly distributive if in (IV), the equality holds. For each subset $\Phi \in P^\star(G)$   where $(G,+,\cdot)$ is the ring of integers  and $\vert \Phi \vert \geq 2$, there exists a multiplicative hyperring $(G_\Phi,+,\circ)$ such that  $G_{\Phi}=G$ and  $x \circ y =\{x.a.y\ \vert \ a \in \Phi\}$ for all $x,y\in G_\Phi$ \cite{das}.
 $\varnothing \neq  A \subseteq G$ is  a  hyperideal  if {\bf 1.}  $x - y \in A$ for any $x, y \in G$,  {\bf 2.} $r \circ x \subseteq A$ for any $x \in A$ and $r \in G$ \cite{f10}.
 A proper hyperideal $A$ of   $G$ is classified as a prime hyperideal if whenever $x,y \in G$ and $x \circ y \subseteq A$, then either $x \in A$ or $y \in A$ \cite{das}.
  For any given hyperideal $A$ of $G$, the prime radical of $A$,  denoted by $rad (A)$, is the intersection of all prime hyperideals of $G$ containing  $A$. If the multiplicative hyperring $G$  has no prime hyperideal containing $A$, then we  define $rad(A)=G$ \cite{das}. A proper hyperideal $A$ of   $G$ refers to  a primary hyperideal if whenever $x,y \in G$ and $x \circ y \subseteq A$, then either $x \in A$ or $y \in rad(A)$ \cite{das}. Here, let us replace $a_1 \circ a_2 \circ \cdots \circ a_n$ with notation  $\bigcirc_{i=1}^n a_i$ for $a_i \in G$. The hyperideal $A$ of $G$ is called  a $\mathcal{C}$-hyperideal  if $\bigcirc_{i=1}^n a_i \cap A \neq \varnothing $ for $a_i \in G$ and $n \in \mathbb{N}$ imply $\bigcirc_{i=1}^n a_i \subseteq A$. Notice that in a multiplicative hyperring $G$, we have
 $\{a \in G \ \vert \  a^n \subseteq A \ \text{for some} \ n \in \mathbb{N}\} \subseteq rad(A)$. By Proposition 3.2 in \cite{das}, we get  $rad(A)=\{a \in G \ \vert \  a^n \subseteq A \ \text{for some} \ n \in \mathbb{N}\}$ if $A$ is a $\mathcal{C}$-hyperideal of $G$. Besides, a hyperideal $A$ of $G$ is  a strong $\mathcal{C}$-hyperideal if $\sum_{i=1}^n(\bigcirc_{j=1}^k x_{ij}) \cap A \neq \varnothing$ for $x_{ij} \in G$ and  $k_i, n \in  \mathbb{N}$, then  $\sum_{i=1}^n(\bigcirc_{j=1}^k x_{ij}) \subseteq  A$. More details concerning strong $\mathcal{C}$-hyperideals were given in  \cite{phd}.   A proper hyperideal $A$ in $G$ is said to be maximal  if for any hyperideal $M$ of $G$ with $A \subset M \subseteq G$, then $M = G$ \cite{ameri}. We denote the intersection of all maximal hyperideals of $G$  by $J(G)$. Moreover, the multiplicative hyperring $G$ is local  if it has just one maximal hyperideal \cite{ameri}.   An element $e$ in $ G$ is considered as an identity element if $a \in e\circ a$ for every  $a \in G$ \cite{ameri}. An element $x$ in $ G$ is called unit, if there exists $y $ in $ G$ such that $e \in y \circ x$.  We denote the set of all unit elements in $G$ by $U(G)$  \cite{ameri}. Also, a hyperring $G$ is a hyperfield if each non-zero element in $G$ is unit. A multiplicative hyperring  $G$ with identity $e$ is  hyperdomain if $0 \in x \circ y$ for  $x,y \in G$, then $x=0$ or $y=0$.
For any give hyperideals  $A_1$ and $A_2$   of $G$, we define $(A_2:A_1)=\{x \in G \ \vert \ x \circ A_1 \subseteq A_2\}$ \cite{ameri}. 
A non-empty subset $S$ of $G$ containing $e$ is called a multiplicative closed subset (briefly, MCS) if $s_1 \circ s_2 \cap S \neq \varnothing$ for every $s_1,s_2 \in S$ \cite{Ghiasvand}. A hyperideal $A$ of $G$ disjoint from $S$ is classified as an $S$-prime hyperideal if there exists a $t \in S$ such that for all $u,v \in G$ if $u \circ v \subseteq A$, then $t \circ u \subseteq A$ or $t \circ v \subseteq A$\cite{Ghiasvand}. There exist many papers devoted to expansions of  prime and primary hyperideals \cite{Ay, far, Ghiasvand2, Mena, Sen,ul}. In fact, these studies aim  to expand the understanding of hyperideals theory in commutative multiplicative hyperrings by  bridging  existing notions and providing  broader framework.

The main purpose of this paper is to introduce the notion of quasi $S$-primary hyperideals in a commutative multiplicative hyperring $G$ as an extention of primary hyperideals where $S$ is an MCS of $G$.  Among many results in this paper, in Section 2, we give many properties and characterizations of this new class of hyperideals. Example \ref{salami} verifies that the notions of  primary hyperideals and quasi $S$-primary hyperideals are different. We obtain that  $A$ is a quasi $S$-primary $\mathcal{C}$-hyperideal of $G$ if and only if $rad(A)$ is an  $S$-prime hyperideal of $G$ in Theorem \ref{11}. In Proposition \ref{intersec}, we investigate whether the intersection  of the collection of quasi $S$-primary hyperideals preserves the algebraic structure. In Theorem \ref{14}, the prime avoidance lemma is presented in analogue form. We describe in detail the behavior of quasi $S$-primary $\mathcal{C}$-hyperideals under homomorphic images and direct product of multiplicative hyperrings. In Example \ref{cart222}, we show that the direct product of  quasi $S_1$-primary hyperideal $A_1$ of $G_1$ and  quasi $S_1$-primary hyperideal $A_2$ of $G_1$ may not be a quasi $S_1 \times S_1$-primary hyperideals of $G_1 \times G_2$. In Section 3, we suggest the idea of weakly quasi $S$-primary hyperideals as an expansion of quasi $S$-primary hyperideals. Example \ref{weak}  shows that a weakly quasi $S$-primary hyperideal may not be a quasi $S$-primary hyperideal. In theorem \ref{21}, we conclude that the radical of a weakly quasi $S$-primary $\mathcal{C}$-hyperideal of $G$ is an $S$-prime hyperideal if $\langle 0 \rangle$ is a quasi $S$-primary $\mathcal{C}$-hyperideal of $G$. Finally, Section 4 is devoted to propose an intermediate class between  $S$-primary hyerideals and quasi $S$-primary hypperideals called strongly quasi $S$-primary hyperideals.   Although, a strongly quasi $S$-primary $\mathcal{C}$-hyperideal $A$ of $G$  is a quasi $S$-primary hyperideal, Theorem \ref{41} shows that the converse holds when $t \circ rad(A)^2 \subseteq A$ for some $t \in S$.  Besides, we give a characterization of strongly quasi-$S$-primary hyperideals in Theorem \ref{42}.

In this paper, we assume that  $A$  is  a commutative multiplicative hyperring and possess the identity element $1$.
\section{  quasi $S$-primary hyperideals }
In \cite{Ghiasvand3}, the authors proposed the notion of $S$-primary hyperideals in context of multiplicative hyperrings, extending primary hyperideals. A hyperideal $A$ of $G$ disjoint from $S$ is classified as an $S$-primary hyperideal if there exists a $t \in S$ such that for all $u,v \in G$ if $u \circ v \subseteq A$, then $t \circ u \subseteq A$ or $t \circ v \subseteq rad(A)$. The aim goal of this section  is to make a contribution to this same research focus, by presenting the
notion of quasi $S$-primary hyperideal.
\begin{definition}
Let $A$ be a  hyperideal of $G$ and $S$ be an MCS of $G$ such that $A \cap S=\varnothing$. $A$ refers to a quasi $S$-primary hyperideal if there exists an element  $t \in S$ such that for all $u,v \in G$ if $u \circ v \subseteq A$, then $t \circ u \subseteq rad(A)$ or $t \circ v \subseteq rad(A)$. In this case, we say that $A$ is associated to $t$.
\end{definition}
 \begin{example} \label{madar} 
Consider the multiplicative hyperring  $(G=\{0,1,2,3\},+,\circ)$ such that 

\[\begin{tabular}{|c|c|c|c|c|} 
\hline $+$ & $0$ & $1$ & $2$ & $3$ 
\\ \hline $0$ & $0$ & $1$ & $2$ & $3$ 
\\ \hline $1$ & $1$ & $2$ & $3$ & $0$ 
\\ \hline$2$ & $2$ & $3$ & $0$ &$1$ 
\\ \hline $3$ & $3$ & $0$ & $1$ & $2$
\\ \hline
\end{tabular}\]
and
\[\begin{tabular}{|c|c|c|c|c|} 
\hline $\circ$ & $0$ & $1$ & $2$ & $3$ 
\\ \hline $0$ & $0$ & $0$ & $0$ & $0$ 
\\ \hline $1$ & $0$ & $I$ & $2$ & $I$ 
\\ \hline$2$ & $0$ & $2$ & $0$ &$2$ 
\\ \hline $3$ & $0$ & $I$ & $2$ & $I$
\\ \hline
\end{tabular}\]
where $I=\{1,3\}$. In this hyperring, $S=\{1,3\}$ is an MCS and $A=\{0,2\}$ is a quasi $S$-primary hyperideal of $G$.
\end{example}
Next, we give an example of a quasi $S$-primary hyperideal that is not a primary hyperideal.
\begin{example} \label{salami}
Consider the polynomial multiplicative hyperring $(\mathbb{Z}_{\Phi}[x],+,\bullet)$ for $\Phi=\{1,3,9\}$ such that 
\[\sum_{\alpha=0}^m a_{\alpha}x^{\alpha}+\sum_{\alpha=0}^n b_{\alpha}x^{\alpha}=\sum_{\alpha=0}^{max\{m,n\}} (a_{\alpha}+b_{\alpha})x^{\alpha}\]
and
\[\sum_{\alpha=0}^m a_{\alpha}x^{\alpha}\bullet\sum_{\alpha=0}^n b_{\alpha}x^{\alpha}=\{\sum_{\alpha=0}^{m+n} c_{\alpha}x^{\alpha} \ \vert \ c_{\alpha} \in \sum_{\beta+\gamma={\alpha}} a_{\beta} \circ b_{\gamma}\}\]
for all $\sum_{\alpha=0}^m a_{\alpha}x^{\alpha}, \sum_{\alpha=0}^n b_{\alpha}x^{\alpha} \in \mathbb{Z}_{\Phi}[x]$. Clearly, $S=\{3^i \ \vert i \geq 0 \}$ is an MCS of $\mathbb{Z}_{\Phi}[x]$. Let $A=\langle 9x \rangle$.   Observe that $A \cap S=\varnothing$. In this hyperring, $A$ is a quasi $S$-primary hyperideal associated to $t=9$. However, $A$ is not a primary hyperideal, because $9 \bullet 2x=\{18x,54x,162x\} \subseteq A$ but $9,2x \notin A$ and $9^n, (2x)^n \notin A$ for all $n \in \mathbb{N}$.
\end{example}
\begin{theorem} \label{11}
Let $A$ be a  $\mathcal{C}$-hyperideal of $G$ and $S$ be an MCS of $G$ such that $A \cap S=\varnothing$. Then  $A$ is a quasi $S$-primary hyperideal of $G$ if and only if  $rad(A)$ is an $S$-prime hyperideal of $G$. 
\end{theorem}
\begin{proof}
$\Longrightarrow$  Let $A$ be a quasi $S$-primary $\mathcal{C}$-hyperideal of $G$. Then there exists a $t \in S$ such that for all $u,v \in G$ if $u \circ v \subseteq A$, then $t \circ u \subseteq rad(A)$ or $t \circ v \subseteq rad(A)$. Assume that $x\circ y \subseteq rad(A)$ for $x,y \in G$. Then we have $x^n \circ y^n \subseteq A$ for some $n \in \mathbb{N}$. Take any $a \in x^n$ and $b \in y^n$. So $a \circ b \subseteq A$. Since  $A$ is a quasi $S$-primary hyperideal of $G$, we obtain $t \circ a \subseteq rad(A)$ or $t \circ b \subseteq rad(A)$.  This implies that $t \circ x^n \subseteq rad(A)$ or $t \circ y^n \subseteq rad(A)$ as $rad(A)$ is a $C$-hyperideal of $G$, $t \circ x^n \cap rad(A) \neq \varnothing$ and $t \circ y^n \cap rad(A) \neq \varnothing$. This means that $t^n \circ x^n \subseteq rad(A)$ or $t^n \circ y^n \subseteq rad(A)$. Therefore $t \circ x \subseteq rad(A)$ or $t \circ y \subseteq rad(A)$. Thus, $rad(A)$ is an $S$-prime hyperideal of $G$. 

$\Longleftarrow$  Suppose that $rad(A)$ is an $S$-prime hyperideal of $G$. Then there exists a $t \in S$ such that for all $u,v \in G$ if $u \circ v \subseteq rad(A)$, then $t \circ u \subseteq rad(A)$ or $t \circ v \subseteq rad(A)$. Let $x \circ y \subseteq A$ for $x,y \in G$. Then we are done as $A \subseteq rad(A)$. 
\end{proof}
The following theorem characterizes the  quasi $S$-primary hyperideals of $G$.
\begin{theorem} \label{12}
Let $A$ be a   $\mathcal{C}$-hyperideal of $G$ and $S$ be an MCS of $G$ such that $A \cap S=\varnothing$. Then the following s are equivalent.

\begin{itemize}
\item[\rm{(i)}]~ $A$ is a quasi $S$-primary hyperideal of $G$.
\item[\rm{(ii)}]~ There exists a  $t \in S$ such that for all hyperideals $A_1,\ldots,A_k$ of $ G$ if $A_1 \circ \cdots\circ A_k \subseteq A$, then $t \circ A_j \subseteq rad(A)$ for some $1 \leq j \leq k$. 
\item[\rm{(iii)}]~ There exists a  $t \in S$ such that for all   $u_1,\ldots,u_k \in G$ if $u_1 \circ \cdots\circ u_k \subseteq A$, then $t \circ u_j \subseteq rad(A)$ for some $1 \leq j \leq k$. 
\end{itemize}
\end{theorem}
\begin{proof}
(i) $\Longrightarrow$ (ii) Assume that $A$ is a quasi $S$-primary hyperideal of $G$. By Theorem \ref{11}, we conclude that $rad(A)$ is an $S$-prime hyperideal of $G$. Then we are done by Corollary 2.11 in \cite{Ghiasvand}.

(ii) $\Longrightarrow$ (iii) Let $u_1 \circ \cdots\circ u_k \subseteq A$ for $u_1,\ldots,u_k \in G$. Assume that $A_i=\langle u_i \rangle$ for any $1 \leq i \leq k$. By the hypothesis, there exists $1 \leq j \leq k$ such that $ t \circ A_j \subseteq rad(A)$ for some $1 \leq j \leq k$ and so $t \circ u_j \subseteq rad(A)$, as needed.

(iii) $\Longrightarrow$ (i) It is sufficient to take $k=2$.
\end{proof}
\begin{proposition}
Let $A$ and $B$ be     $\mathcal{C}$-hyperideals of $G$ and $S$ be an MCS of $G$ such that $B \cap S \neq \varnothing$. If $A$ is a quasi $S$-primary hyperideal of $G$, then so are $B \circ A$ and $A \cap B$.
\end{proposition}
\begin{proof}
Since $B \circ A \subseteq  A$ and $S \cap A = \varnothing$, we get $(B \circ A) \cap S = \varnothing$. Suppose that $A$ is a quasi $S$-primary hyperideal of $G$ and $A$ is associated to $t$.    Let $u \circ v \subseteq B \circ A \subseteq A$ for $u,v \in G$. Then we have $t \circ u \subseteq rad(A)$ or $t \circ v \subseteq rad(A)$. Since $B \cap S \neq \varnothing$, we take any $s \in B \cap S$. Hence we obtain  $s \circ t \circ u \subseteq B \circ rad(A) \subseteq rad(B) \circ rad(A)=rad(B \circ A)$ or $s \circ t \circ v \subseteq B \circ rad(A) \subseteq rad(B \circ A)$. Since $S$ is a MCS and $t,s \in S$, there exists $r \in (s \circ t) \cap S$. Therefore, we have $r \circ u  \subseteq rad(B \circ A)$ or $r \circ v  \subseteq rad(B \circ A)$. This shows that $B \circ A$ is a quasi $S$-primary hyperideal of $G$. By a similar argument, one can prove that $A \cap B$ is a quasi $S$-primary hyperideal of $G$.
\end{proof}
\begin{proposition}
Let $S_1 \subseteq S_2$ be two multiplicatively closed subsets of $G$ such that for  every $t \in S_2$,  there exits $s \in S_1$ with $(t \circ s) \cap  S_1 \neq \varnothing $ and $A$ be a $\mathcal{C}$-hyperideal of $G$ such that $A \cap S_2 =\varnothing$. Then $A$ is a quasi $S_1$-primary hyperideal of $G$ if and only if $A$ is a quasi $S_2$-primary hyperideal of $G$.
\end{proposition}
\begin{proof}
$\Longrightarrow$ It is straightforward.

$\Longleftarrow $ Let $A$ be a quasi $S_2$-primary hyperideal of $G$ associated to $t_2 \in S_2$. By the hypothesis,  there exists $t_1 \in S_1$ such that $(t_1 \circ t_2) \cap S_1 \neq \varnothing$. Assume that $u \circ v \subseteq A$ for $u,v \in G$. Therefore $t_2 \circ u \subseteq rad(A)$ or $t_2 \circ v \subseteq rad(A)$ as $A$ is a quasi $S_2$-primary hyperideal of $G$ associated to $t_2 \in S_2$. Choose $t \in (t_1 \circ t_2)  \cap S_1$. Hence, we obtain $t \circ u \subseteq t_1 \circ t_2 \circ u \subseteq rad(A)$ or $t \circ v \subseteq t_1 \circ t_2 \circ v \subseteq rad(A)$. Thus, $A$ is a quasi $S_1$-primary hyperideal of $G$.
\end{proof}
However,  the intersection of a family of quasi $S$-primary hyperideals is not quasi $S$-primary  but, if the radical of the hyperideals are the same,  we have the following result.
\begin{proposition}  \label{intersec}
Let $A_1, \ldots,A_k$ be quasi $S$-primary $\mathcal{C}$-hyperideals of $G$. If $rad(A_i)=rad(A_j)$ for each $i,j \in \{1,\ldots,k\}$, then $A=\cap_{i=1}^kA_i$ is a quasi $S$-primary  hyperideal of $G$.
\end{proposition}
\begin{proof}
From  $A_i \cap S = \varnothing$ for all $1 \leq i \leq k$, it follows that $A \cap S = \varnothing$. Assume that $A_i$ is associated to $t_i$ for each $1 \leq i \leq k$. Choose $t \in (t_1 \circ \cdots \circ t_k) \cap S$. Assume that $u \circ v \subseteq A$ for $u,v \in G$ such that $t \circ u \nsubseteq rad(A)$. This implies that $t_{\alpha} \circ u \nsubseteq rad(A_{\alpha})$ for some $1 \leq \alpha \leq k$ which  means $t_{\alpha} \circ v \subseteq rad(A_{\alpha})$ as $A_{\alpha}$ is a quasi $S$-primary hyperideals of $G$, $u \circ v \subseteq A_{\alpha}$ and $t_{\alpha} \circ u \nsubseteq rad(A_{\alpha})$. Therefore we conclude that $t \circ v  \subseteq rad(A_{\alpha})$. Since $A_1, \ldots,A_k$ are hyperideals of $G$ with the same radical, we obtain $t \circ v \subseteq \cap_{i=1}^k rad(A_i)$ and so $t \circ v \subseteq rad(A)$. Consequently, $A=\cap_{i=1}^kA_i$ is a quasi $S$-primary  hyperideal of $G$.
\end{proof}
We next give  other characterisations of quasi $S$-primary hyperideals.
\begin{theorem}
Let $A$ be a   $\mathcal{C}$-hyperideal of $G$ and $S$ be an MCS of $G$ such that $A \cap S=\varnothing$. Then the following s are equivalent.
\begin{itemize}
\item[\rm{(i)}]~ $A$ is a quasi $S$-primary hyperideal of $G$ and there exists a  $r \in S$ such that for every   $a \in G$ if $r \circ a^2 \subseteq A$, then $r \circ a \subseteq A$. 
\item[\rm{(ii)}]~ $A$ is an $S$-prime hyperideal of $G$.
\item[\rm{(iii)}]~ $A$ is an $S$-primary hyperideal of $G$ and there exists a  $r \in S$ such that for every   $a \in G$ if $r \circ a^2 \subseteq A$, then $r \circ a \subseteq A$.  
\end{itemize}
\end{theorem}
\begin{proof}
(i) $\Longrightarrow$ (ii) Suppose that $A$ is a quasi $S$-primary hyperideal of $G$. Then  there exists a $t \in S$ such that for all $u,v \in G$ if $u \circ v \subseteq A$, then $t \circ u \subseteq rad(A)$ or $t \circ v \subseteq rad(A)$. Since $S$ is an MCS,  we take any $s \in (r \circ t) \cap S$. Assume that $u \circ v \subseteq A$ for $u,v \in G$.  Let $t \circ u \subseteq rad(P)$. Then  there exists $\alpha \in \mathbb{N}$ such that $(t \circ u)^{\alpha} \subseteq P$. Now, we consider two cases. Case 1.  $\alpha=2n$ for some $n \in \mathbb{N}$. Take any $x \in  (t \circ u)^n$. So $ r \circ x^2 \subseteq A$ which implies  $r \circ x \subseteq A$. Since $A$ is a $\mathcal{C}$-hyperideal of $G$ and $r \circ (t \circ u)^n \cap A \neq \varnothing$, we have $r \circ (t \circ u)^n \subseteq A$. Case 2.    $\alpha=2n+1$ for some $n \in \mathbb{N}$. So we get  $(t \circ u)^{2n+2} \subseteq A$. Take any $y \in  (t \circ u)^{n+1}$. Therefore  $ r \circ y^2 \subseteq A$ which means  $r \circ y \subseteq A$. Since $r \circ (t \circ u)^{n+1} \cap A \neq \varnothing$ and $A$ is a $\mathcal{C}$-hyperideal of $G$, we have $r \circ (t \circ u)^{n+1} \subseteq A$. In these two cases, $n$ is an even or odd number. By continuing this process, we get $r \circ (t \circ u)^2 \subseteq A$ or $r \circ (t \circ u)^3 \subseteq A$ and we finally obtain $s \circ  u \subseteq r \circ t \circ u \subseteq A$. This implies that $A$ is an $S$-prime hyperideal of $G$. Now, let $t \circ v \subseteq rad(P)$. By a similar argument, we get $s \circ  v \subseteq r \circ t \circ v \subseteq A$, as inquired. 

(ii) $\Longrightarrow$ (iii) $\Longrightarrow$  (i) are clear. 
\end{proof}

\begin{proposition} \label{13}
Let $A$ be a   $\mathcal{C}$-hyperideal of $G$ and $S$ be an MCS of $G$ such that $A \cap S=\varnothing$. Then $A$ is a quasi $S$-primary hyperideal of $G$ if and only if there exists an element $t \in S$ such that $(rad(P) : t)$ is a prime hyperideal of $G$.
\end{proposition}
\begin{proof}
The claim follows from  Proposition 3.6 in  \cite{Ghiasvand} and Theorem \ref{11}.
\end{proof}
In the following, we conclude a result which is similar to of the prime avoidance lemma  presented in the context of  multiplicative hyperrings in \cite{das}.
\begin{theorem} \label{14}
Let $A_1,\ldots,A_k$ be quasi $S$-primary $\mathcal{C}$-hyperideals of $G$ and $A$ be a hyperideal of $G$. If $A$ is  contained in $\cup_{i=1}^k A_i$, then there exists $t \in S$   such that $t \circ A \subseteq rad(A_j)$ for some $1 \leq j \leq k$.
\end{theorem}
\begin{proof}
Assume that the hyperideal $A$ of $G$ is contained in  $\cup_{i=1}^k A_i$ where $A_i$ is a quasi $S$-primary $\mathcal{C}$-hyperideal of $G$ for each $1 \leq i \leq k$. Then  for any $1 \leq i \leq k$ there exists a $t_i \in S$ such that for all $u,v \in G$ if $u \circ v \subseteq A_i$, then $t_i \circ u \subseteq rad(A_i)$ or $t_i \circ v \subseteq rad(A_i)$. By Proposition \ref{13}, $(rad(A_i):t_i)$ is a prime hyperideal of $G$. Since $A$ is contained in $\cup_{i=1}^k(rad(A_i):t_i)$, there exists $1 \leq j \leq k$ such that  $A \subseteq (rad(A_j):t_j)$ by Proposition 2.19 in \cite{das}. This implies that $t_j \circ A \subseteq rad(A_j)$.
\end{proof}
By using Theorem \ref{14}, we can deduce easily the following
corollary.
\begin{corollary}
Let $A$ be a proper hyperideal of $G$. Then $A$ is a prime hyperideal of $G$ if and only if $A$ is a quasi primary hyperideal of $G$ and $a^2 \subseteq A$ for $a \in G$ implies $a \in A$.
\end{corollary}
\begin{proof}
It suffices to take $S=\{1\}$ in Theorem \ref{14}.
\end{proof}
For any given multiplicative hyperring $G$,  $M_m(G)$ denotes the set of all hypermatrices of $G$. Assume that $I = (I_{ij})_{m \times m}, J = (J_{ij})_{m \times m} \in P^\star (M_m(G))$. Then $I_{ij} \subseteq J_{ij}$ if and only if $I \subseteq J$ \cite{ameri}. 
\begin{theorem}
Let $A$ be a   hyperideal of $G$ and $S$ be an MCS of $G$ such that $A \cap S=\varnothing$. If $M_m(A)$ is a quasi $M_m(S)$-primary  $\mathcal{C}$-hyperideal of $M_m(G)$ where $M_m(S)=\{\begin{pmatrix}
t & 0 & \cdots & 0 \\
0 & 0 & \cdots & 0 \\
\vdots & \vdots & \ddots \vdots \\
0 & 0 & \cdots & 0 
\end{pmatrix} \ \vert \ t \in S\}$, then $A$ is a quasi $S$-primary hyperideal of $G$. 
\end{theorem}
\begin{proof}
Let $u \circ v \subseteq A$ for $u,v \in G$. Then we have
\[\begin{pmatrix}
u \circ v & 0 & \cdots & 0 \\
0 & 0 & \cdots & 0 \\
\vdots & \vdots & \ddots \vdots \\
0 & 0 & \cdots & 0 
\end{pmatrix}
\subseteq M_m(A).\]
Since $M_m(A)$ is a quasi $M_m(S)$-primary   hyperideal of $M_m(G)$ and 
\[ \begin{pmatrix}
u \circ v & 0 & \cdots & 0\\
0 & 0 & \cdots & 0\\
\vdots & \vdots & \ddots \vdots\\
0 & 0 & \cdots & 0
\end{pmatrix}
=
\begin{pmatrix}
u & 0 & \cdots & 0\\
0 & 0 & \cdots & 0\\
\vdots & \vdots & \ddots \vdots\\
0 & 0 & \cdots & 0
\end{pmatrix}
\circ  
\begin{pmatrix}
v & 0 & \cdots & 0\\
0 & 0 & \cdots & 0\\
\vdots & \vdots & \ddots \vdots\\
0 & 0 & \cdots & 0
\end{pmatrix},
\]
there exists $\begin{pmatrix}
t & 0 & \cdots & 0 \\
0 & 0 & \cdots & 0 \\
\vdots & \vdots & \ddots \vdots \\
0 & 0 & \cdots & 0 
\end{pmatrix} \in M_m(S)$ such that  
\[ \begin{pmatrix}
t & 0 & \cdots & 0\\
0 & 0 & \cdots & 0\\
\vdots& \vdots & \ddots \vdots\\
0 & 0 & \cdots & 0
\end{pmatrix} 
\circ 
\begin{pmatrix}
u & 0 & \cdots & 0\\
0 & 0 & \cdots & 0\\
\vdots& \vdots & \ddots \vdots\\
0 & 0 & \cdots & 0
\end{pmatrix}
=
\begin{pmatrix}
t \circ u & 0 & \cdots & 0\\
0 & 0 & \cdots & 0\\
\vdots& \vdots & \ddots \vdots\\
0 & 0 & \cdots & 0
\end{pmatrix}
\subseteq rad(M_m(A))\]\\
or 
\[ \begin{pmatrix}
t & 0 & \cdots & 0\\
0 & 0 & \cdots & 0\\
\vdots& \vdots & \ddots \vdots\\
0 & 0 & \cdots & 0
\end{pmatrix} 
\circ 
\begin{pmatrix}
v & 0 & \cdots & 0\\
0 & 0 & \cdots & 0\\
\vdots& \vdots & \ddots \vdots\\
0 & 0 & \cdots & 0
\end{pmatrix}
=
\begin{pmatrix}
t \circ v & 0 & \cdots & 0\\
0 & 0 & \cdots & 0\\
\vdots& \vdots & \ddots \vdots\\
0 & 0 & \cdots & 0
\end{pmatrix}
\subseteq rad(M_m(A)).\] 
In the first possibility, there exists $n \in \mathbb{N}$ such that 
\[\begin{pmatrix}
(t \circ u)^n & 0 & \cdots & 0\\
0 & 0 & \cdots & 0\\
\vdots& \vdots & \ddots \vdots\\
0 & 0 & \cdots & 0
\end{pmatrix}=
\begin{pmatrix}
t \circ u & 0 & \cdots & 0\\
0 & 0 & \cdots & 0\\
\vdots& \vdots & \ddots \vdots\\
0 & 0 & \cdots & 0
\end{pmatrix}^n \subseteq M_m(A)\]
which means $(t \circ  u)^n  \subseteq A$ and so $t \circ  u \subseteq rad(A)$. In the second possibility, in the same way we get  $t \circ  v \subseteq rad(A)$. Thus we conclude that  $A$ is a quasi $S$-primary hyperideal of $G$. 
\end{proof}
Recall from \cite{f10} that a mapping $\eta$ from the   multiplicative hyperring
$(G_1, +_1, \circ _1)$ into the   multiplicative hyperring $(G_2, +_2, \circ _2)$ is a hyperring good homomorphism if $\eta(a +_1 b) =\eta(a)+_2 \eta(b)$,  $\eta(a\circ_1 b) = \eta(a)\circ_2 \eta(b)$ for every $a,b \in A_1$ and $\eta(1_{G_1})=\eta(1_{G_2})$.

\begin{theorem} \label{homo} 
Let $G_1$ and $G_2$ be two  multiplicative hyperrings such that the mapping $\eta$ from $ G_1$ into $ G_2$ is a hyperring
good homomorphism, $A_1$ and $A_2$ $\mathcal{C}$-hyperideals of $G_1$ and $G_2$, respectively,  and $S$   an MCS of $G_1$ such that $0_{G_2} \notin \eta(S)$. Then the followings are satisfied:
\begin{itemize}
\item[\rm{(i)}]~ If $A_2$ is a  quasi $\eta(S)$-primary  hyperideal of $G_2$, then $\eta^{-1}(A_2)$ is a quasi $S$-primary  hyperideal  of $G_1$.
\item[\rm{(ii)}]~ If  $A_1$ is a quasi $S$-primary  hyperideal of $A_1$ with $Ker (\eta) \subseteq A_1$ and $\eta$ is surjective, then $\eta(A_1)$ is a quasi $\eta(S)$-primary  hyperideal of $G_2$.
\end{itemize}
\end{theorem}
\begin{proof}
(i) Let $\eta^{-1} (A_2) \cap S \neq \varnothing$. Then there exists $t \in \eta^{-1} (A_2) \cap S$ which means $\eta(t) \in A_2 \cap S$ which is impossible. Therefore we have $\eta^{-1} (A_2) \cap S= \varnothing$. Now, assume that $u \circ_1 v \subseteq \eta^{-1}(A_2)$ for $u,v \in G_1$. Then we obtain $\eta(u) \circ_2 \eta(v)=\eta(u \circ _1 v) \subseteq A_2$   as $\eta$ is a good homomorphism. Since $A_2$ is a  quasi $\eta(S)$-primary  hyperideal of $G_2$, there exists $t \in S$ such that $\eta(t) \circ_2 \eta(u)=\eta(t \circ_1 u ) \subseteq rad(A_2)$  or $\eta(t) \circ_2 \eta(v)=\eta(t \circ_1 v ) \subseteq rad(A_2)$. This implies that $t \circ_1 u \subseteq \eta^{-1}(rad(A_2))  $ or $t \circ_1 v \subseteq \eta^{-1}(rad(A_2))  $. Then $t \circ_1 u   \subseteq rad(\eta^{-1}(A_2))$ or $t \circ_1 v   \subseteq rad(\eta^{-1}(A_2))$ by Lemma 2.8 in \cite{Akray}. Thus, $\eta^{-1}(A_2)$ is a quasi $S$-primary  hyperideal  of $G_1$.

(ii) If $s \in \eta(A_1) \cap \eta(S)$, then there exist $a \in A_1$ and $t \in S$ such that $s=\eta(a)=\eta(t)$. Then we have $\eta(a-t)=0$ which means $a-t \in Ker (\eta) \subseteq A_1$ and so $t \in A_1$ which is impossible. Let $u_2 \circ_2 v_2 \subseteq \eta(A_2)$ for $u_2, v_2 \in G_2$. Since $\eta$ is surjective, there exist $u_1, v_1 \in G_1$ such that $\eta(u_1)=u_2$ and $\eta(v_1)=v_2$. Then, we get $\eta(u_1 \circ v_1)=\eta(u_1) \circ_2 \eta(v_1) \subseteq \eta(A_1)$. Take any $w \in u_1 \circ v_1$. Hence $\eta(w) \in \eta(A_1)$. This yields that there exists $a \in A_1$ such that $\eta(w)=\eta(a)$. Then  $w - a \in Ker (\eta) \subseteq A_1$ and so $w \in A_1$. Since $A_1$ is a $\mathcal{C}$-hyperideal of $G_1$ and $(u_1 \circ v_1) \cap A_1 \neq \varnothing$, we obtain $u_1 \circ v_1 \subseteq A_1$. By the hypothesis,  there exists $t \in S$ such that $t \circ u_1 \subseteq rad(A_1)$ or $t \circ v_1 \subseteq rad(A_1)$. This implies that $\eta(t \circ u_1) \subseteq \eta(rad(A_1))$ or $\eta(t \circ v_1) \subseteq \eta(rad(A_1))$. Hence $ \eta(t) \circ \eta(u_1)= \eta(t) \circ u_2 \subseteq rad(\eta(A_1))$ or $ \eta(t) \circ \eta(v_1)= \eta(t) \circ v_2 \subseteq rad(\eta(A_1))$. Consequently, $\eta(A_1)$ is a quasi $\eta(S)$-primary  hyperideal of $G_2$.
\end{proof}
Now, we establish the following corollaries.
\begin{corollary}
Let $G \subseteq H$ be an extension of $G$ and $S$ be an MCS of $G$. If $A$ be a quasi $S$-primary  $\mathcal{C}$-hyperideal $H$, then $A \cap G$ is a quasi $S$-primary  hyperideal $G$.
\end{corollary}
\begin{proof}
Consider the monomorphism $\eta: G \longrightarrow H$ defined by $\eta(x)=x$ for all $x \in G$. Then we are done by Theorem \ref{homo} (i). 
\end{proof}
 Assume that $B$ is a hyperideal of $G$.  Then quotient abelian group $G/B=\{x+B \ \vert \ x \in G\}$ becomes a hyperring with the multiplication $(x+B) \circ (y+B)=\{a+B \ \vert \ a \in x \circ y\}$. In this case, $G/B$ is said to be quotient hyperring \cite{Sen}. Let $S$ be a MCS of $G$ such that $B \cap S= \varnothing$. Let $t \in S$. Then the equivalence class
of $t$ in the quotient hyperring $G/B$ is denoted by $\bar{t}=t+S$. Suppose that $\bar{S}=\{t+B \ \vert \ t \in S\}$ and $t_1+B, t_2+B \in \bar{S}$. Since $(t_1 \circ t_2)  \cap S \neq \varnothing$, we have $(t_1+B)\circ  (t_2+B ) \cap \bar{S} \neq \varnothing$. This implies that $\bar{S}$ is a MCS of $G/B$ \cite{Ghiasvand}.
Now, we have the following result.
\begin{corollary}
Let $B$ be a hyperideal of $G$, $A$     a $\mathcal{C}$-hyperideal  of $G$ containing   $B$ and $S$    an MCS of $G $ such that $A \cap S=\varnothing$. If $A$ is a quasi $S$-primary  hyperideal of $G$, then $A/B$   is a quasi $\bar{S}$-primary  hyperideal of $G/B$.
\end{corollary}
\begin{proof}
Since $A \cap S=\varnothing$, we have $A/B \cap \bar{S}=\varnothing$. Consider the  epimorphism $\eta :G \longrightarrow G/B$ defined by $\eta(x)=x+B$. Now,  the claim follows from Theorem \ref{homo}. 
\end{proof}

Let $(G_1,+_1,\circ_1)$ and $(G_2,+_2,\circ_2)$ be two multiplicative hyperrings with nonzero identity. The set $G_1 \times G_2$ with the operation $+$ and the hyperoperation $\circ$ defined as

$(u_1,u_2)+(v_1,v_2)=(u_1+_1v_1,u_2+_2v_2)$

$(u_1,u_2) \circ (v_1,v_2)=\{(u,v) \in G_1 \times G_2 \ \vert \ u \in u_1 \circ_1 v_1, v \in u_2 \circ_2 v_2\}$ \\
is a multiplicative hyperring \cite{ul}.  Now, allow us to identify all  quasi $S$-primary hyperideals within the direct product of two cross products.
\begin{theorem} \label{cart}
Assume that $A=A_1 \times A_2$ such that    $A_1$ and $A_2$ are $C$-hyperideals of the commutative multiplicative hyperrings $G_1$ and $G_2$, respectivelty and $S=S_1 \times S_2$ such that  $S_1$ and $S_2$ are multiplicative closed subsets of $G_1$ and $G_2$, respectively. Then the following conditions are equivalent.
\begin{itemize}
\item[\rm{(i)}]~  $A$ is a quasi $S$-primary hyperideal of $G=G_1 \times G_2$.
\item[\rm{(ii)}]~ $S_1 \cap A_1 \neq \varnothing$ and $A_2$ is a quasi $S_2$-primary hyperideal of $G_2$ or $S_2 \cap A_2 \neq \varnothing$ and $A_1$ is a quasi $S_1$-primary hyperideal of $G_1$.
\end{itemize}
\end{theorem}
\begin{proof}
$ \Longrightarrow $ Let $A$ be a quasi $S$-primary hyperideal of $G=G_1 \times G_2$ associated to $(t_1,t_2)$ and $S_1 \cap A_1=S_2 \cap A_2= \varnothing$. Assume that $(u,v) \in A$. Since   $A_1$ and $A_2$ are $C$-hyperideals of $G_1$ and $G_2$, respectively, and $(u,1_{G_2}) \circ (1_{G_1},v) \cap A \neq \varnothing$, we have $(u,1_{G_2}) \circ (1_{G_1},v) \subseteq A$. By the hypothesis, we conclude that $  (t_1,t_2) \circ (u,1_{G_2}) \subseteq rad(A)$ or $  (t_1,t_2) \circ (1_{G_1},v) \subseteq rad(A)$ and so $(t_1 \circ_1 u, t_2 \circ_2 1_{G_2})   \subseteq rad(A)$ or $(t_1 \circ_1 1_{G_1}, t_2 \circ_2 v)  \subseteq rad(A)$. Then there exists $m,n \in \mathbb{N}$ such that $t_1^m \subseteq A_1$ or $t_2^n \subseteq A_2$. Since $t_1^m \cap S_1 \neq \varnothing$ and $t_2^n \cap S_2 \neq \varnothing$, we get $S_1 \cap A_1 \neq \varnothing$ or $S_2 \cap A_2 \neq \varnothing$. Let us assume that $S_1 \cap A_1 \neq \varnothing$. Since $S \cap A =\varnothing$, we get $S_2 \cap A_2 =\varnothing$. Assume that $u_2 \circ_2 v_2 \subseteq A_2$ for $u_2,v_2 \in G_2$. Take $t \in S_1 \cap A_1$. Then $(t,u_2) \circ (1_{G_1},v_2) \subseteq A$. Therefore we have $(t_1,t_2) \circ (t,u_2) \subseteq rad(A)$ or $(t_1,t_2) \circ (1_{G_1},v_2) \subseteq rad(A)$. This implies that $t_2 \circ_2 u_2 \subseteq  rad(A_2)$ or $t_2 \circ_2 v_2 \subseteq  rad(A_2)$. Thus, $A_2$ is a quasi $S_2$-primary hyperideal of $G_2$.

$ \Longleftarrow $  Let us assume that $S_1 \cap A_1 \neq \varnothing$ and $A_2$ is a quasi $S_2$-primary hyperideal of $G_2$ associated to $t_2$. We take $t_1 \in S_1 \cap A_1$. Now, let $(u_1,u_2)  \circ (v_1.v_2) \subseteq A$ for $(u_1,u_2) ,  (v_1.v_2) \in G$. Therefore, we have $u_2 \circ v_2 \subseteq A_2$. By the hypothesis, we obtain $t_2 \circ_2 u_2 \subseteq rad(A_2)$ or $t_2 \circ_2 v_2 \subseteq rad(A_2)$. This implies that $(t_1 \circ_1 u_1,t_2 \circ_2 u_2) \subseteq rad(A)$ or $(t_1 \circ_1 v_1,t_2 \circ_2 v_2) \subseteq rad(A)$. Hence $(t_1,t_2) \circ (u_1,u_2) \subseteq rad(A)$ or $(t_1,t_2) \circ (v_1,v_2) \subseteq rad(A)$. Consequently, $A$ is a quasi $S$-primary hyperideal of $G$.
\end{proof}
The following example verifies that if $A_1$ and $A_2$ are quasi $S_1$-primary hyperideal and quasi $S_2$-primary hyperideal of $G_1 $ and $ G_2$, respectively, then $A_1 \times A_2$ may not be a quasi $S_1 \times S_2$-primary hyperideal of $G_1 \times G_2$. This means that the condition ``$S_1 \circ A_1 \neq \varnothing$" 
( ``$S_2 \circ A_2 \neq \varnothing$")  is important in Theorem \ref{cart}.
\begin{example} \label{cart222}
Consider the multiplicative hyperrings $G_1=(\mathbb{Z}_{\phi_1},+,\circ_1)$ and $G_2=(\mathbb{Z}_{\phi_2},+,\circ_2)$ where $\Phi_1=\{1,2,4\}$ and $\Phi_2=\{1,3,9\}$. Note that $S_1=\{2^i \vert \ i\geq 0\}$ and $S_2=\{3^i \vert \ i \geq 0\}$ are multiplicatively closed subsets  of $G_1$ and $G_2$, respectively. It is easy to see  $A_1=\langle 3 \rangle$ and $A_2=\langle 2 \rangle$ are   quasi $S_1$-primary hyperideal and quasi $S_2$-primary hyperideal of $G_1 $ and $ G_2$, respectively.  But, $A_1 \times A_2$ is not   quasi $S_1 \times S_2$-primary hyperideal of $G_1 \times G_2$ since $(1,2) \circ (3,1)=\{(a,b) \ \vert \ a \in \{3,6,12\}, b \in \{2,6,18\}\} \subseteq A_1 \times A_2$ but neither $(t_1,t_2) \circ (1,2) =\{(a,b) \ \vert \ a \in \{ t_1,2t_1,4t_1\}, b \in \{2t_2,6t_2,18t_2\}\} \subseteq rad(A_1 \times A_2)$     nor  $(t_1,t_2) \circ (3,1) =\{(a,b) \ \vert \ a \in \{3 t_1,6t_1,12t_1\}, b \in \{ t_2,3t_2,9t_2\}\} \subseteq rad(A_1 \times A_2)$ for all $(t_1,t_2) \in S_1 \times S_2$.
\end{example}
In view of Theorem \ref{cart}, we conclude the following result.
\begin{corollary}
Assume that $A=A_1 \times \cdots \times A_k$ such that    $A_i$ is a $C$-hyperideal of the commutative multiplicative hyperring  $G_i$ for each $1 \leq i \leq k$  and $S=S_1 \times \cdots \times S_k$ such that  $S_i$ is  an MCS of $G_i$ for all $1 \leq i \leq k$. Then the following conditions are equivalent.
\begin{itemize}
\item[\rm{(i)}]~  $A$ is a quasi $S$-primary hyperideal of $G=G_1 \times \cdots \times  G_k$.
\item[\rm{(ii)}]~   $A_j$ is a quasi $S_j$-primary hyperideal of $G_j$ for some $1 \leq j \leq k$ and $S_i \cap A_i \neq \varnothing$ for all $1 \leq i \leq k$ such that $i \neq j$. 
\end{itemize}
\end{corollary}
\section{weakly quasi $S$-primary hyperideals}
In his paper \cite{Guesmi}, Guesmi defined the notion of weakly quasi $S$-primary ideals in a commutative ring, generalizing both weakly $S$-primary and quasi $S$-primary ideals.  In this section, we introduce and invetigate the ``weakly" analog of quasi $S$-primary hyperideals. We aim to thoroughly understand their defining characteristics by presenting various theorems and proofs. We begin with the deﬁnition.
\begin{definition}
Let $A$ be a  hyperideal of $G$ and $S$ be an MCS of $G$ such that $A \cap S=\varnothing$. We say that $A$ is a weakly quasi $S$-primary hyperideal if there exists an element  $t \in S$ such that for all $u,v \in G$ if $0 \notin u \circ v \subseteq A$, then $t \circ u \subseteq rad(A)$ or $t \circ v \subseteq rad(A)$. In this case, we say that $A$ is associated to $t$.
\end{definition}
\begin{example} \label{weak}
Let  $G=(\mathbb{Z}_6, +, \circ )$. We define $x \circ y =\{xy,2xy,3xy,4xy,5xy\}$ for all $x,y \in G$. In this multiplicative hyperring $S=\{1,5\}$ is a MCS. It is clear that $\langle 0 \rangle$ is a weakly quasi $S$-primary hyperideal of $G$. However, since $2 \circ 3 \subseteq  \langle 0 \rangle$ but $2 \circ t \notin rad(\langle 0 \rangle)$ and $3 \circ t \notin rad( \langle 0 \rangle)$ for all $t \in S$,  $\langle 0 \rangle$ is not a quasi $S$-primary hyperideal of $G$.
\end{example}
In Theorem \ref{11}, it was shown that the radical of a quasi $S$-primary $\mathcal{C}$-hyperideal of $G$ is an $S$-prime hyperideal of $G$. In the following theorem, we show that  the radical of a weakly quasi $S$-primary $\mathcal{C}$-hyperideal is an $S$-prime hyperideal of $G$ when $\langle 0 \rangle$ is a quasi $S$-primary $\mathcal{C}$-hyperideal of $G$.
\begin{theorem} \label{21}
Let $A$ be a  $\mathcal{C}$-hyperideal of $G$ and $S$ be an MCS of $G$ such that $A \cap S=\varnothing$. If $A$ is a weakly quasi $S$-primary hyperideal of $G$ and $\langle 0 \rangle$ is a quasi $S$-primary $\mathcal{C}$-hyperideal, then $rad(A)$ is an $S$-prime hyperideal of $G$. 
\end{theorem}
\begin{proof}
Assume that $u \circ v \subseteq rad(A)$ for $u,v \in G$, the weakly quasi $S$-primary hyperideal  $A$ is associated to $t$ and the quasi $S$-primary hyperideal $\langle 0 \rangle$ is associated to $s$. Therefore we have $(u \circ v )^n \subseteq A$ for some $n \in \mathbb{N}$. Take any $x \in u^n$ and $y \in v^n$. Let  $0 \notin x \circ y$. Since $A$ is a weakly quasi $S$-primary hyperideal of $G$ and $0 \notin x \circ y \subseteq A$,  we get  $t \circ x \subseteq rad(A)$ or $t \circ y \subseteq rad(A)$. Since $rad(A)$ is a  $\mathcal{C}$-hyperideal of $G$, we obtain $ (t \circ u) ^n \subseteq rad(A)$ because $ t \circ u^n  \cap rad(A) \neq \varnothing$ or we have $ (t \circ v )^n \subseteq rad(A)$ because $ t \circ v^n  \cap rad(A) \neq \varnothing$. This implies that $t \circ u \subseteq rad(A)$ or $t \circ v \subseteq rad(A)$. Now, let $0 \in x \circ y$. Then $x \circ y \subseteq \langle 0 \rangle$. Since $\langle 0 \rangle$ is a quasi $S$-primary hyperideal of $G$, we get $s \circ x \subseteq  rad(\langle 0 \rangle)  \subseteq rad(A)$ or $s \circ y \subseteq rad(\langle 0 \rangle) \subseteq rad(A)$. Since $rad(A)$ is a  $\mathcal{C}$-hyperideal of $G$, we get $ (s \circ u) ^n \subseteq rad(A)$ because $ s \circ u^n  \cap rad(A) \neq \varnothing$ or we obtain $ (s \circ v )^n \subseteq rad(A)$ because $ s \circ v^n  \cap A \neq \varnothing$. Hence we conclude that $s \circ u \subseteq rad(A)$ or $s \circ v \subseteq rad(A)$. Since $S$ is a MCS and $s,t \in S$, there exists $r \in (s \circ t) \cap S$. Therefore, we get the reult that $r \circ u \subseteq s \circ t \circ u \subseteq rad(A)$ or $r \circ v \subseteq s \circ t \circ v \subseteq rad(A)$. Thus, $rad(A)$ is an $S$-prime hyperideal of $G$  associated to $r$.
\end{proof}
The next corollary is an immediate consequence of Theorem \ref{21}.
\begin{corollary} \label{22}
Let $A$ be a  $\mathcal{C}$-hyperideal of $G$ and $S$ be an MCS of $G$ such that $A \cap S=\varnothing$. If $A$ is a weakly quasi  primary hyperideal of $G$ and $\langle 0 \rangle$ is a quasi  primary $\mathcal{C}$-hyperideal, then $rad(A)$ is an prime hyperideal of $G$. 
\end{corollary}
\begin{proof}
Take $S = \{1\}$ in Theoerm \ref{21}.
\end{proof}
Next, we present some equivalent statements to characterize weakly quasi  $S$-primary hyperideals in a multiplicative hyperring.
\begin{theorem} \label{23}
Let $A$ and $\langle 0 \rangle$ be  $\mathcal{C}$-hyperideals of $G$ and  $S$ be an MCS of $G$ such that $A \cap S=\varnothing$. Then the followings are equivalent.
\begin{itemize}
\item[\rm{(i)}]~$A$ is a weakly quasi  $S$-primary hyperideal of $G$.
\item[\rm{(ii)}]~ There exists $t \in S$ such that $(A:x)=(0:x)$ or $(A:x) \subseteq (rad(A) : t)$ for all $x \notin (rad(A) : t)$. 
\item[\rm{(iii)}]~There exists $t \in S$ such that for hyperideal  $B$ of $G$ and $x \in G$, $0 \neq x \circ B \subseteq A$ implies $t \circ x \subseteq rad(A)$ or $t \circ B \subseteq rad(A)$. 
\item[\rm{(iv)}]~There exists $t \in S$ such that for hyperideals  $B,D$ of $G$, $0 \neq B \circ D \subseteq A$ implies  $t \circ B \subseteq rad(A)$ or $t \circ D \subseteq rad(A)$. 
\end{itemize}
\end{theorem}
\begin{proof}
(i) $\Longrightarrow$ (ii) Let $A$ be a weakly quasi  $S$-primary hyperideal of $G$ associated to $t$ such that $(A:x) \neq (0:x)$ for $x \in G \backslash(rad(A) : t)$. Then there exists $y \in G$ such that $0 \notin x \circ y \subseteq A$ as $\langle 0 \rangle$ is a  $\mathcal{C}$-hyperideal of $G$. By the hypothesis, we have $t \circ x \subseteq rad(A)$ or $t \circ y \subseteq rad(A)$. Then we get $t \circ y \subseteq rad(A)$ as $x \in G \backslash(rad(A) : t)$. Assume that $a \in (A : x)$ such that $0 \notin a \circ x$. Since $A$ is a weakly quasi  $S$-primary hyperideal of $G$, $0 \notin a \circ x \subseteq A$ and  $x \in G \backslash(rad(A) : t)$, we obtain $t \circ a \subseteq rad(A)$ which means $a \in (rad(P) : t)$. Now, let us assume that $a \in (A : x)$ such that $0 \in a \circ x$. Since $0 \notin x \circ (y +a) \subseteq x \circ y + x \circ a \subseteq x \circ y +\langle 0 \rangle= x \circ y \subseteq A$ and $x \in G \backslash(rad(A) : t)$, we obtain $t \circ (y+a) \subseteq rad(A)$. Since $rad(A)$ is a $\mathcal{C}$-hyperideal of $G$ and $t \circ (y+a) \subseteq t \circ y + t \circ a$, we get $t \circ y + t \circ a \subseteq rad(A)$ which means $t \circ  a \subseteq rad(A)$ as $t \circ y \subseteq rad(A)$. Thus, $a \in (rad(A) : t)$ and so $(A:x) \subseteq (rad(A) : t)$.

(ii) $\Longrightarrow$ (iii) Let for hyperideal  $B$ of $G$ and $x \in G$, $0 \neq x \circ B \subseteq A$ and $t \circ x \nsubseteq rad(A)$ for $t \in S$ as in (ii). By the hypothesis, we obtain $(A:x)=(0:x)$ or $(A:x) \subseteq (rad(A) : t)$. From  $(A:x)=(0:x)$, it follows that $x \circ B =0$, a contradiction. Then we have $(A:x) \subseteq (rad(A) : t)$ which implies that $t \circ B \subseteq rad(A)$. 

(iii) $\Longrightarrow$ (iv) Assume that for hyperideals  $B,D$ of $G$, $0 \neq B \circ D \subseteq A$ and  $t \circ B \nsubseteq rad(A)$ for $t \in S$ as in (iii). Then we conclude that $t \circ x \nsubseteq rad(A)$ for some $x \in B$.  Let $x \circ D \neq 0$. We obtain $t \circ D \subseteq rad(A)$, by the assumption. Now, let us assume that $x \circ D=0$. Then there exists $y \in B$ such that $y \circ D \neq 0$ as $B \circ D \neq 0$. Suppose that $t \circ y \nsubseteq rad(A)$. By the hypothesis, we get  $t \circ D \subseteq rad(A)$. If $t \circ y \subseteq rad(A)$, then $t \circ (x+y) \subseteq t \circ x + t \circ y \nsubseteq rad(A)$ as $t \circ x \nsubseteq rad(A)$. Since $0 \neq (x+y) \circ D \subseteq A$ and $t \circ (x+y) \notin rad(A)$, we have $t \circ D \subseteq rad(A)$, as needed.

(iv) $\Longrightarrow$ (iii) Assume that $0 \notin u \circ v \subseteq A$ for some $u,v \in A$. Then we have $0 \neq \langle u \rangle  \circ \langle v \rangle \subseteq \langle u \circ v  \rangle \subseteq A$. By the hypothesis, we get $t \circ \langle u \rangle \subseteq rad(A)$ or $t \circ \langle v  \rangle \subseteq rad(A)$. Hence we conclude that $t \circ u \subseteq rad(A)$ or $t \circ v \subseteq rad(A)$. Consequently, $A$ is a weakly quasi  $S$-primary hyperideal of $G$.
\end{proof}
\begin{theorem} \label{24}
Let $A$ and $\langle 0 \rangle$ be     $\mathcal{C}$-hyperideals of $G$ and $S$ be an MCS of $G$ such that $A \cap S=\varnothing$. 
\begin{itemize}
\item[\rm{(i)}]~ If $A$ is a  weakly quasi $S$-primary hyperideal such that  $A^2 \neq \langle 0 \rangle$, then $A$ is a quasi $S$-primary hyperideal of $G$. 
\item[\rm{(ii)}]~ If  $A$ is a  weakly quasi $S$-primary hyperideal  of $G$   that   is not a quasi $S$-primary hyperideal of $G$, then $rad(A)=rad(\langle 0 \rangle)$.
\end{itemize}
\end{theorem}
\begin{proof}
(i) Let $A$ be a  weakly quasi $S$-primary hyperideal of $G$ associated to $t$, $u \circ v \subseteq A$ for $u,v \in G$. Let  $0 \notin u \circ v $. Then we have $t \circ u \subseteq rad(A)$ or $t \circ v \subseteq rad(A)$ as  $A$ is a  weakly quasi $S$-primary hyperideal of $G$. Now, let $0 \in  u \circ v $. We consider three cases. Case I. Assume that $u \circ A \neq 0$. Then there exists $a \in A$ such that $ u \circ a \neq 0$.  Hence $0 \notin u \circ (a+v) \subseteq u \circ a + u \circ v \subseteq A$ as $\langle 0 \rangle$  is    a $\mathcal{C}$-hyperideals of $G$. Since $A$ is a  weakly quasi $S$-primary hyperideal of $G$ associated to $t$, we have $t \circ u \subseteq rad(A)$ or $t \circ (a+v) \subseteq rad(A)$. In the second possibility, we have $t \circ v \subseteq rad(A)$ as $(t \circ a+ t \circ v) \cap rad(A) \neq \varnothing$ and $rad(A)$ is a $\mathcal{C}$-hyperideals of $G$. Case II. Let $v \circ A \neq 0$. Then there exists $b \in A$ such that $ v \circ b \neq 0$. Therefore $0 \notin (b+u) \circ v \subseteq b \circ v + u \circ v \subseteq A$ as $\langle 0 \rangle$  is    a $\mathcal{C}$-hyperideals of $G$. By the hypothesis, we obtain $t \circ (b+u) \subseteq rad(A)$ or $t \circ v \subseteq rad(A)$. In the former possibility, we get $t \circ u \subseteq rad(A)$ as $rad(A)$ is a $\mathcal{C}$-hyperideals of $G$ and $(t \circ b+ t \circ u) \cap rad(A) \neq \varnothing$. Case III. Suppose that $u \circ A=v \circ A=0$. By the assumption, we have $x \circ y \neq 0$ for some $x,y \in A$. It follows that $0 \notin   (u+x) \circ (v+y) \subseteq u \circ v+u \circ y+x \circ v+x \circ y \subseteq A $. Then we conclude that $t \circ (u+x) \subseteq rad(A)$ which means $t \circ u \subseteq rad(A)$ or $t \circ (v+y) \subseteq rad(A)$ which implies $t \circ v \subseteq rad(A)$. Consequently, $A$ is a quasi $S$-primary hyperideal of $G$.

(ii) Since $A$ is a  weakly quasi $S$-primary hyperideal  of $G$ that   is not a quasi $S$-primary hyperideal of $G$, we have $A^2=\langle 0 \rangle$ by (i). This implies that $u^2 \subseteq \langle 0 \rangle$ for each $u \in A$ which means $A \subseteq rad(\langle 0 \rangle)$. Then we obtain $rad(A)=rad(\langle 0 \rangle)$ as $rad(A) \subseteq rad(\langle 0 \rangle)$.
\end{proof}
For any given hyperideal $A$ of $G$, letting $O_A=\{x \in A \ \vert \ x \in x \circ A\}$. Recall from \cite{ameri5} that $A$ is a pure hyperideal if $A=O_A$. Now, we give the following theorem as a result of Theorem \ref{24}. 
\begin{theorem}
Let $\langle 0 \rangle$ be  a   $\mathcal{C}$-hyperideals of $G$   and $A$ be a  weakly quasi $S$-primary hyperideal  of $G$   that   is not a quasi $S$-primary hyperideal of $G$. If $A$ is a pure hyperideal of $G$, then $A=\langle 0 \rangle$.
\end{theorem}
\begin{proof}
Let $u \in A$. Since $A$ is a pure hyperideal of $G$, we have $u \in u \circ v $ for some $v \in A$. Since $A$ is a  weakly quasi $S$-primary hyperideal  of $G$   that   is not a quasi $S$-primary hyperideal of $G$, $A^2=0$ by Theorem \ref{24} (1). This means that $u \in \langle 0 \rangle$ and so $A=\langle 0 \rangle$.
\end{proof}
\section{strongly quasi $S$-primary hyperideals}
In this section, strongly quasi $S$-primary hyperideals as an intermediate class between $S$-primary hyperideals defined in \cite{Ghiasvand3} and quasi $S$-primary hyperideals   is introduced. We then delve into the properties and characterizations of them. First we give the deﬁnition.
\begin{definition}
Let $A$ be a  hyperideal of $G$ and $S$ be an MCS of $G$ such that $A \cap S=\varnothing$. We define $A$ as a strongly quasi $S$-primary hyperideal if there exists an element  $t \in S$ such that for all $u,v \in G$ if $  u \circ v \subseteq A$, then $t \circ u^2 \subseteq A$ or $t \circ v \subseteq rad(A)$. In this case, we say that $A$ is associated to $t$.
\end{definition}
 \begin{example}
 In Example \ref{weak}, the hyperideals $A_1=\{0,3\}$ and $A_2=\{0,2,4\}$ are strongly quasi $S$-primary hyperideals of $G$ but the hyperideal $\langle 0 \rangle$ is not a strongly quasi $S$-primary hyperideals of $G$. 
 \end{example}
 \begin{example} \label{haji}
 Consider the multiplicative hyperring $(G_{\Phi},+,\circ)$ where $G=\mathbb{Z}_5$ and $\Phi=\{1,2,3\}$. In this hyperring, $\langle 0 \rangle$ is both strongly quasi $S$-primary $G$ and quasi $S$-primary hyperideal of $G$ where $S=\{1,3\}$ and then it is a weakly quasi $S$-primary hyperideal of $G$.
 \end{example}
 \begin{theorem} \label{41}
 Let $A$ be a   $\mathcal{C}$-hyperideal of $G$ and $S$ be an MCS of $G$ such that $A \cap S=\varnothing$ and $t \circ rad(A)^2 \subseteq A$ for some $t \in S$. Then $A$ is a strongly $S$-quasi hyperideal of $G$ if and only if $A$ is a quasi $S$-primary hyperideal of $G$.
 \end{theorem}
 \begin{proof}
 It is clear that every strongly $S$-quasi hyperideal of $G$ is a quasi $S$-primary hyperideal of $G$. Conversely, assume that $A$ is a quasi $S$-primary hyperideal of $G$ associated to $s \in S$. Let $u \circ v \subseteq A$ for $u,v \in G$. Then we have $s \circ u \subseteq rad(A)$ or $s \circ v \subseteq rad(A)$. Hence,  we have $t\circ (s \circ u)^2 \subseteq t \circ rad(A)^2 \subseteq A$ or $t \circ s^2 \circ v \subseteq rad(A)$. Since $S$ is a MCS of $G$, there exists $w \in (t \circ s^2) \cap S$. This implies that $w \circ u^2 \subseteq A$ or $w \circ v \subseteq rad(A)$. Thus, $A$ is a strongly $S$-quasi hyperideal of $G$.
 \end{proof}
 Our next theorem gives several equivalent conditions for a hyperdeal $A$ disjoint with $S$ to be strongly quasi $S$-primary.

 \begin{theorem} \label{42}
 Let $A$ be a   $\mathcal{C}$-hyperideal of $G$ and $S$ be an MCS of $G$ such that $A \cap S=\varnothing$. Then the followings are quivalent:
 \begin{itemize}
\item[\rm{(i)}]~ $A$ is a strongly quasi $S$-primary  hyperideal of $G$.
\item[\rm{(ii)}]~$A$ is a strongly quasi $\mathfrak{S}$-primary  hyperideal of $G$ where $\mathfrak{S}=\{a \in G \ \vert \ (a \circ b) \cap S \neq  \varnothing  \ \text{for some} \ b \in G\}$.
\item[\rm{(iii)}]~There exists an $t \in S$ such that for every $u \in G$, $u^2 \subseteq (A: t)$ or $(A : u) \subseteq (rad(A) : t)$. 
\item[\rm{(iv)}]~There exists an $t \in S$ such that for all hyperideals $B, C$ of $A$, if $B \circ C \subseteq A$, then $t \circ \mathfrak{B} \subseteq A$ or $t \circ C \subseteq rad(A)$ where $\mathfrak{B}=\cup_{b \in B} b^2 $.
\item[\rm{(v)}]~There exists an $t \in S$ such that for every $u \in G$, $u^n  \subseteq (A : t)$ for some $n \in \mathbb{N}$ or $\mathfrak{D} \subseteq (A:t)$ where $\mathfrak{D}=\cup_{d \in (A: u)} d^2$.
 \end{itemize}
 \end{theorem}
 \begin{proof}
 (i) $\Longrightarrow$ (ii) Let $\mathfrak{S}=\{a \in G \ \vert \ (a \circ b) \cap S \neq  \varnothing  \ \text{for some} \ b \in G\}$ and $a \in \mathfrak{S} \cap A$. Then there exists $b \in G$ such that $(a \circ b) \cap S \neq \varnothing$. Take any $c \in (a \circ b) \cap S$. Since $a \in \mathfrak{S} \cap A$, we get  $c \in A \cap S$ which is impossible. By the hypothesis and the fact that $S \subseteq \mathfrak{S}$, we conclude that $A$ is a strongly quasi $\mathfrak{S}$-primary  hyperideal of $G$.
 
  (ii) $\Longrightarrow$ (iii) Let $A$ be a strongly quasi $\mathfrak{S}$-primary  hyperideal of $G$ associated to $\mathfrak{t}$. Then there exists $b \in G$ such that $(b \circ \mathfrak{t} ) \cap S \neq \varnothing$. Choose $t \in (b \circ \mathfrak{t} ) \cap S$. Let $u^2 \nsubseteq (A : t)$ for some $u \in G$ and $ v \in (A : u)$. Since  $A$ is  a strongly quasi $\mathfrak{S}$-primary  hyperideal of $G$ ssociated to $\mathfrak{t}$ and $u \circ v \subseteq A$, we obtain $\mathfrak{t} \circ u^2 \subseteq A$ or $\mathfrak{t} \circ v \subseteq rad(A)$. This implies that $t \circ u^2 \subseteq b \circ \mathfrak{t}  \circ u^2 \subseteq A$ or $t \circ v \subseteq b \circ \mathfrak{t} \circ v \subseteq rad(A)$. Since the first case is impossible, we conclude that $t \circ v \subseteq rad(A)$ and so $(A : u) \subseteq (rad(A) : t)$.
 
   (iii) $\Longrightarrow$ (iv) Suppose that $B \circ C \subseteq A$ for hyperideals $B, C$ of $G$. Take the fixed $t$ in (iii) and assume that $t \circ C \nsubseteq rad(A)$. Then we have $t \circ c \nsubseteq rad(A)$ for some $c \in C$. Let $u \in B$. Then we conclude that $c \in (A : u)$ but $c \notin(rad(A) : t)$ which implies $u^2 \subseteq (A:t)$ by (iii), as needed. 
  
   (iv) $\Longrightarrow$ (v) Put $B=(A: u)$ and $C=\langle u \rangle$ for $u \in G$. Since $B \circ C  \subseteq  A$, we get $t \circ \mathfrak{D} \subseteq A$ or $t  \circ \langle u \rangle \subseteq rad(A)$ where $\mathfrak{D}=\cup_{b \in (A:u)} b^2 $. In the second case, we have $t^n \circ u^n \subseteq A$ for some $n \in \mathbb{N}$. Choose $s \in t^n \cap S$. Therefore we obtain $s \circ \mathfrak{D}  \subseteq A$ or $s \circ u^n \subseteq  A$. Thus, $\mathfrak{D}  \subseteq (A:s)$ or $u^n \subseteq (A:s)$.
  
    (v) $\Longrightarrow$ (i) Take the fixed $t$ in (v). Assume that $u \circ v \subseteq A$ for some $u,v \in G$ such that $t \circ v \nsubseteq rad(A)$.  This implies that $t^n \circ v^n \nsubseteq A$ for all $n \in \mathbb{N}$ and so $v^n \notin (A : t)$. Since $u \in (A: v)$, we get $u^2 \subseteq (A:s)$ by (v) and so $s \circ u^2 \subseteq  A$. Consequently,   $A$ is a strongly quasi $S$-primary  hyperideal of $G$.
 \end{proof}
Now, we examine quasi $S$-primary hyperideals
within the context of hyperring homomorphisms.
 \begin{theorem} \label{homo2} 
Let   $A_1$ and $A_2$ be $\mathcal{C}$-hyperideals of multiplicative hyperrings $G_1$ and $G_2$, respectively,  and the mapping $\eta: G_1\longrightarrow G_2$  be a hyperring
good homomorphism. Let $S$ be an MCS of $G_1$ such that $0_{G_2} \notin \eta(S)$. Then the followings hold:
\begin{itemize}
\item[\rm{(i)}]~ If $A_2$ is a  strongly quasi $\eta(S)$-primary  hyperideal of $G_2$, then $\eta^{-1}(A_2)$ is a strongly quasi $S$-primary  hyperideal  of $G_1$.
\item[\rm{(ii)}]~ If  $A_1$ is a strongly quasi $S$-primary  hyperideal of $A_1$ with $Ker (\eta) \subseteq A_1$ and $\eta$ is surjective, then $\eta(A_1)$ is a strongly quasi $\eta(S)$-primary  hyperideal of $G_2$.
\end{itemize}
\end{theorem}
\begin{proof}
(i) Clearly, $\eta^{-1}(A_2) \cap S = \varnothing$. Let $u \circ_1 v \subseteq \eta^{-1}(A_2)$ for $u,v \in G_1$. Since $\eta$ is a good homomorphism, we get $\eta(u) \circ_2 \eta(v)=\eta(u \circ _1 v) \subseteq A_2$   . Since $A_2$ is a  strongly quasi $\eta(S)$-primary  hyperideal of $G_2$, there exists $t \in S$ such that $\eta(t) \circ_2 \eta(u)^2=\eta(t \circ_1 u^2 ) \subseteq A_2$  or $\eta(t) \circ_2 \eta(v)=\eta(t \circ_1 v ) \subseteq rad(A_2)$. This means that $t \circ_1 u^2 \subseteq \eta^{-1}(A_2)  $ or $t \circ_1 v \subseteq \eta^{-1}(rad(A_2))  \subseteq rad(\eta^{-1}(A_2))$. Hence, $\eta^{-1}(A_2)$ is a strongly quasi $S$-primary  hyperideal  of $G_1$.

(ii) It is easy to verify that $\eta(A_1) \cap \eta(S)=\varnothing$. Assume that  $u_2 \circ_2 v_2 \subseteq \eta(A_2)$ for $u_2, v_2 \in G_2$. Then there exist $u_1, v_1 \in G_1$ such that $\eta(u_1)=u_2$ and $\eta(v_1)=v_2$ as $\eta$ is surjective. Therefore, we have $\eta(u_1 \circ v_1)=\eta(u_1) \circ_2 \eta(v_1) \subseteq \eta(A_1)$. Choose $a \in u_1 \circ v_1$. Therefore $\eta(a) \in \eta(A_1)$. This implies that there exists $b \in A_1$ such that $\eta(a)=\eta(b)$. Then  $a - b \in Ker (\eta) \subseteq A_1$ and so $a \in A_1$. Since $(u_1 \circ v_1) \cap A_1 \neq \varnothing$ and $A_1$ is a $\mathcal{C}$-hyperideal of $G_1$, we have $u_1 \circ v_1 \subseteq A_1$. Then, there exists $t \in S$ such that $t \circ u_1^2 \subseteq  A_1 $ or $t \circ v_1 \subseteq rad(A_1)$ as $A_1$ is a strongly quasi $S$-primary  hyperideal of $G_1$. This means that $\eta(t) \circ u_2^2= \eta(t) \circ \eta(u_1^2)= \eta(t \circ u_1^2) \subseteq \eta (A_1)$ or $ \eta(t) \circ v_2=\eta(t) \circ \eta(v_1)=\eta(t \circ v_1) \subseteq \eta(rad(A_1))\subseteq rad(\eta(A_1))$.  Consequently, $\eta(A_1)$ is a strongly quasi $\eta(S)$-primary  hyperideal of $G_2$.
\end{proof}
Let us give the following theorem without proof as a result of Theorem \ref{homo2}.
\begin{corollary}
Let $B$ be a hyperideal of $G$, $A$     a $\mathcal{C}$-hyperideal  of $G$ containing   $B$,  $  H$  an extension of $G$ and $S$  an MCS of $G$.
\begin{itemize}
\item[\rm{(i)}]~  If $A$ be a strongly quasi $S$-primary  hyperideal $H$, then $A \cap G$ is a strongly quasi $S$-primary  hyperideal $G$.

\item[\rm{(ii)}]~    If $A$ is a strongly quasi $S$-primary  hyperideal of $G$, then $A/B$   is a strongly quasi $\bar{S}$-primary  hyperideal of $G/B$.
\end{itemize}
\end{corollary}
Next, we characterize strongly quasi $S$-primary hyperideals in a direct product of multiplicative hyperrings.
\begin{theorem} \label{cart}
Let   $S_1$ and $S_2$ be multiplicative closed subsets of  the commutative multiplicative hyperrings $G_1$ and $G_2$, respectivelty and,   $A_1$ and $A_2$ be $C$-hyperideals $G_1$ and $G_2$, respectively. Then $A_1 \times A_2$ is a strongly quasi $S_1 \times S_2$-primary hyperideal of $G_1 \times G_2$  if and only if   $A_2$ is a strongly quasi $S_2$-primary hyperideal of $G_2$ and $S_1 \cap A_1 \neq \varnothing$  or  $A_1$ is a strongly quasi $S_1$-primary hyperideal of $G_1$ and $S_2 \cap A_2 \neq \varnothing$.
\end{theorem}
\begin{proof}
$ \Longrightarrow $  Let $A_1 \times A_2$ be  a strongly quasi $S$-primary hyperideal of $G_1 \times G_2$ associated to $(t_1,t_2)$. It is easy to verify that   $S_1 \cap A_1 \neq \varnothing$ or $S_2 \cap A_2 \neq  \varnothing$. Let us assume that $S_1 \cap A_1 \neq \varnothing$. Since $S_1 \times S_2 \cap A_1 \times A_2 =\varnothing$, we have  $S_2 \cap A_2 =\varnothing$. Suppose  that $u_2 \circ_2 v_2 \subseteq A_2$ for $u_2,v_2 \in G_2$. Choose $s \in S_1 \cap A_1$. Therefor  $(s,u_2) \circ (1_{G_1},v_2) \subseteq A_1 \times A_2$. Then we get $(t_1,t_2) \circ (s,u_2)^2 \subseteq  A_1 \times A_2 $ or $(t_1,t_2) \circ (1_{G_1},v_2) \subseteq rad(A_1 \times A_2)$. This means that $t_2 \circ_2 u_2^2 \subseteq   A_2 $ or $t_2 \circ_2 v_2 \subseteq  rad(A_2)$. Consequently, $A_2$ is a strongly quasi $S_2$-primary hyperideal of $G_2$.

$ \Longleftarrow $  Suppose that  $A_2$ is a strongly quasi $S_2$-primary hyperideal of $G_2$ associated to $t_2$ and $S_1 \cap A_1 \neq \varnothing$.  Choose $t_1 \in S_1 \cap A_1$. Assume that  $(u_1,u_2)  \circ (v_1.v_2) \subseteq A_1 \times A_2$ for $(u_1,u_2) ,  (v_1.v_2) \in G_1 \times G_2$. Then, we get $u_2 \circ v_2 \subseteq A_2$. Therefore we conclude that  $t_2 \circ_2 u_2^2 \subseteq  A_2 $ or $t_2 \circ_2 v_2 \subseteq rad(A_2)$. This means  that $(t_1 \circ_1 u_1^2,t_2 \circ_2 u_2^2) \subseteq  A_1 \times A_2 $ or $(t_1 \circ_1 v_1,t_2 \circ_2 v_2) \subseteq rad(A_1 \times A_2)$. It follows that  $(t_1,t_2) \circ (u_1,u_2)^2 \subseteq A_1 \times A_2$ or $(t_1,t_2) \circ (v_1,v_2) \subseteq rad(A_1 \times A_2)$. Thus, $A$ is a strongly quasi $S$-primary hyperideal of $G_1 \times G_2$.
\end{proof}
As a result of the previous   theorem, we give the following corollary.
\begin{corollary}
Let  $S_i$ be an MCS of the commutative multiplicative hyperring $G_i$ for   $1 \leq i \leq k$ and     $A_i$ be a $C$-hyperideal of   $G_i$ for   $1 \leq i \leq k$. Then   $ A_1 \times \cdots \times A_k$  is a strongly quasi $S_1 \times \cdots \times S_k$-primary hyperideal of $ G_1 \times \cdots \times  G_k$ if and only if   $A_j$ is a quasi $S_j$-primary hyperideal of $G_j$ for some $1 \leq j \leq k$ and $S_i \cap A_i \neq \varnothing$ for all $1 \leq i \leq k$ such that $i \neq j$. 
\end{corollary}

\section{conclusion}
This study centered around three generalizations of primary hyperideals, namely quasi $S$-primary, weakly quasi $S$-primary and  strongly quasi $S$-primary hyperideals where $S$ is an MCS of $G$.  The results given in this paper contribute to a better realizing of the structure of hyperrings. These findings are crucial  for
further research in this field and serve as a  foundation for future studies. 
Although every primary hyperideal is a quasi $S$-primary, we verified the converse need not to be hold. We examined the radical of a quasi $S$-primary $\mathcal{C}$-hyperideals and showed that it is an  $S$-prime hyperideal of $G$. We gave  a result which is similar to of the prime avoidance lemma in the context of multiplicative hyperrings. We investigated  quasi  $S$-primary $\mathcal{C}$-hyperideals within the context of hyperring homomorphisms. Moreover, many basic properties and characterizations of weakly quasi $S$-primary hyperideals were provided. We showed that the notions of quasi $S$-primary hyperideals and weakly quasi $S$-primary hyperideals are different. We concluded that the radical of a weakly quasi $S$-primary $\mathcal{C}$-hyperideal is an $S$-prime hyperideal when the hyperideal $0$ is a quasi $S$-primary $\mathcal{C}$-hyperideal.
At the end of the paper, a subclass of quasi $S$-primary hyperideals called strongly quasi $S$-primary hyperideals was studied. We gave  some  specific results explaining this   structure. We identified all strongly quasi $S$-primary $\mathcal{C}$-hyperideals within the direct product of two cross products. As a new research subject, we suggest the notion  of graded quasi $S$-primary hyperideals of a graded multiplicative hyperring.


\end{document}